\documentclass{gtpart}     
\title{On Phantom Maps and Finiteness Conditions}

\author{J. Schwass}
\givenname{James}
\surname{Schwass}
\address{Department of Mathematics\\ Grand Valley State University\newline{} Allendale, MI 49001}
\email{james.schwass@gvsu.edu}
\urladdr{}


\arxivreference{}
\arxivpassword{}

\volumenumber{}
\issuenumber{}
\publicationyear{}
\papernumber{}
\startpage{}
\endpage{}
\doi{}
\MR{}
\Zbl{}
\received{}
\revised{}
\accepted{}
\published{}
\publishedonline{}
\proposed{}
\seconded{}
\corresponding{}
\editor{}
\version{}

\usepackage{amsmath,amsfonts,amsthm,amssymb,verbatim,enumerate}
\usepackage[all]{xy}
\usepackage{setspace}
\usepackage{Tabbing}
\usepackage{fancyhdr}
\usepackage{lastpage}
\usepackage{extramarks}
\usepackage{chngpage}
\usepackage{soul,color}
\usepackage{graphicx,float,wrapfig}
\usepackage[toc,page]{appendix}

\renewcommand{\tilde}{\widetilde}

\DeclareMathOperator{\aut}{\mathrm{aut}}
\DeclareMathOperator{\Aut}{\mathrm{Aut}}

\renewcommand{\phi}{\varphi}

 \newcommand{\A}        {\mathcal{A}}
  \newcommand{\B}{\mathcal{B}}

    \newcommand{\F}        {\mathcal{F}}
        \newcommand{\I}{\mathcal{I}}
        \newcommand{\N}        {\mathcal{N}}
\renewcommand{\P}{\mathcal{P}}
\renewcommand{\R}  {\mathcal{R}}
 \newcommand{\ZZ}    {\mathbb{Z}}
    \newcommand{\RR}    {\mathbb{R}}
    \newcommand{\CC}    {\mathbb{C}}
    \newcommand{\NN}    {\mathbb{N}}
 \newcommand{\cp}    {\CC\mathrm{P}}

    \newcommand{\om}    {\Omega}
    \newcommand{\s}     {\Sigma}

\newcommand{\Top}{\mathbf{Top}}

\newcommand{\lra}[1]{\stackrel{#1 }{\longrightarrow}}

\DeclareMathOperator{\holim}{holim}

            \DeclareMathOperator{\Cl}    {cl}
      \DeclareMathOperator{\pl}    {pl}

        \DeclareMathOperator{\colim} {colim}
        
        \DeclareMathOperator{\conn}{conn}

 \DeclareMathOperator{\Ph}    {Ph}

        \DeclareMathOperator{\map}    {map}

\newtheorem{thm}{Theorem}[section]    
\newtheorem{cor}{Corollary}[section]
\newtheorem{prop}{Proposition}[section]
\newtheorem*{zlem}{Zabrodsky Lemma}      
\newtheorem{ex}{Example}[section]
\theoremstyle{definition}
\newtheorem*{rem}{Remark}             
\begin{document}

\begin{abstract}    
A phantom map is a potentially nontrivial map which induces the zero map on every homology theory and on homotopy groups. 
Zabrodsky has shown that in the presence of particular finiteness conditions on spaces $X$ and $Y$ \emph{every} map $X\to Y$  is a phantom map. More specifically, Zabrodsky essentially requires $Y$ to be a finite CW complex and $X$ to be a Postnikov space. We show Zabrodsky's observations hold under less restrictive finiteness conditions on the spaces $X$ and $Y$, making use of the Zabrodsky lemma and the machinery of resolving classes. 

As an application we identify, up to extension, the group of self-homotopy equivalences of spaces belonging to a particular family. 
\end{abstract}

\maketitle


\section{Introduction}

The purpose of this work is to further our understanding of finiteness conditions through the lens of the theory of phantom maps, and to examine applications of this understanding to calculations in (classical) homotopy theory. Traditional finiteness conditions include having the homotopy type of a finite CW complex, or dually that of a Postnikov space. 
By a Postnikov space, we mean a space $X$ which is homotopy equivalent to its $N$th Postnikov approximation, $X^{(N)}$, for some $N$. In other words, the Postnikov tower 
\[
\dots \to X^{(n)}\to X^{(n-1)}\to \dots
\]
stabilizes, and could be truncated to a tower of finite-length without loss of data. 

A generalization of the notion of a Postnikov space is that of a polyGEM, in the sense of Dwyer and Farjoun \cite{DwyerFarjoun}; a space $X$ is a polyGEM if it is built up from generalized Eilenberg-MacLane spaces (GEMs) in a finite number of principle fibrations. Such a space has a finite length Generalized Postnikov Tower in the sense of Iriye and Kishimoto \cite{IriyeKishimoto}. These spaces were originally called oriented polyGEMs by Farjoun in \cite{farjoun1995cellular}. 
Roughly dual to the notion of a polyGEM is that of an $\F$-finite space: a space is $\F$-finite if it can be built from finite-type wedges of spheres in a finite number of principal cofibrations. These are the spaces of finite $\F$-cone length in the sense of Arkowitz, Stanley, and Strom \cite{ASS}, where $\F$ is the collection of finite-type wedges of spheres. We note for future reference that an $\F$-finite space is necessarily of finite-type as is a polyGEM. We wondered how differently $\F$-finite spaces can behave from finite spaces from the perspective of phantom map theory, and dually for polyGEMs and Postnikov spaces. 


A map $f:X\to Y$ is called a \textbf{phantom map} if for each map $H\to X$ with $H$ a finite CW complex the composite $H\to X \to Y$ is nullhomotopic. 
Phantom maps are interesting for several reasons; these maps are invisible to an algebraic topologists standard toolkit of homology theories and homotopy groups, but carry topologically nontrivial information, and play an important role in homotopy theory for this reason. Phantom maps have been used by Gray and others to construct counterintuitive examples, as in the theory of Same $N$-Type (SNT) sets \cite{GraySNT}.
Roitberg \cite{RoitbergAut} and others have used phantom map theory to make calculations in the study of the groups $\Aut(X)$ of homotopy classes of self-homotopy equivalences of a space $X$. 
We will present a similar application of our findings below.

A lucrative way to study phantom maps $X\to Y$ between finite-type nilpotent spaces $X$ and $Y$ centers on the profinite completion $Y\to  Y^\wedge$. It is not difficult to show that a map $X\to Y$ is a phantom map if and only if the composite $X\to Y \to Y^\wedge$ is nullhomotopic (see \cite{handbook} or \cite{zabrodsky}, for example). So, we can describe $\Ph(X,Y)$, the subset of $[X,Y]$ consisting of homotopy classes of phantom maps, as the weak kernel of the natural map 
\[
[X,Y]\to [X, Y^\wedge]. 
\]
When $\map_\ast(X, Y^\wedge)\sim \ast$, it follows readily that $[X,Y]=\Ph(X,Y)$. Zabrodsky \cite{zabrodsky} has shown that under the condition $\map_\ast(X, Y^\wedge)\sim \ast$ there is a natural bijection of pointed sets 
\begin{equation}
\Ph(X,Y)\cong \prod H^k(X,\pi_{k+1}(Y)\otimes \RR)\label{PhantomIdentCool}.
\end{equation}
In case $X$ is a cogroup or $Y$ is grouplike, this bijection is an isomorphism of groups. 
For these reasons, the condition $\map_\ast(X, Y^\wedge)\sim \ast$ is of significant interest in the theory of phantom maps.

The following theorem of Zabrodsky, which is a consequence of Miller's confirmation of the Sullivan conjecture, completely characterizes maps from a Postnikov space into a finite CW complex. 



\begin{thm}\label{thmZabrodsky}\cite{zabrodsky} 
If $K$ is a simply-connected Postnikov space of finite-type
and $Y$ is a finite CW complex, then for any $n,m\geq 0$
$\map_\ast(\s^n K, \om^m Y^\wedge)\sim \ast.$
In particular, every map $\s^n K\to \om^m Y$ is a phantom map. 
\end{thm}

Theorem \ref{thmZabrodsky} gives us an opportunity to demonstrate, for the skeptical, that essential (i.e. non-nullhomotopic) phantom maps exist. Since $\cp^\infty=K(\ZZ,2)$ we have 
\[
[\cp^\infty,S^3]=\Ph(\cp^\infty,S^3)\cong \prod H^k(\cp^\infty;\pi_{k+1}(S^3)\otimes\RR) \cong \RR,
\]
and so there are uncountably many distinct homotopy classes of phantom maps $\cp^\infty\to S^3$. This example is originally due to Gray \cite{GraySNT}, though the calculation there has a decidedly different flavor.

In this note we prove the following generalization of Theorem \ref{thmZabrodsky}. 


\begin{thm}\label{thmMain}
If $K$ is a simply-connected polyGEM and $Y$ is $\F$-finite,  then for any $n,m\geq 0$
$\map_\ast(\s^n K, \om^m Y^\wedge)\sim \ast.$
\end{thm}

In light of \eqref{PhantomIdentCool} we have the following immediate corollary of Theorem \ref{thmMain}. 

\begin{cor}\label{corPhantomIdent}
If $K$ is a simply-connected polyGEM of finite type and $Y$ is an $\F$-finite space, then for any $n,m\geq 0$
\[
[\s^n K,\om^m Y]= \Ph(\s^n K,\om^m Y)\cong \prod_n H^n(\s^n K;\pi_{n+1}(\om^k Y)\otimes\RR). 
\]
\end{cor} 

A common way to work with polyGEMs ($\F$-finite spaces) is through a form of induction on the structural (co)fiber sequences defining such spaces. At first glance one might expect this technique to be ill suited for the task at hand, since we have no formal reason to expect the functors $\map_\ast(K,-)$ and $\map_\ast(-,Y^\wedge)$ to produce predictable results when applied to cofiber and fiber sequences, respectively. 
Nonetheless, by way of two duality violating mechanisms, the Zabrodsky lemma and the machinery of resolving classes, we obtain the result.

In Section \ref{secBackground} we give detailed descriptions of the terms in the statement of Theorem \ref{thmMain}. We also recall the statement of the Zabrodsky lemma and give a few details on resolving classes. The proof of Theorem \ref{thmMain} is given in the Section \ref{secMainResult}.

Theorem \ref{thmMain} seems to indicate that polyGEMs behave similarly to Postnikov spaces as domains of phantom maps, and $\F$-finite spaces behave similarly to finite spaces as targets of phantom maps. One might be curious if these similarities persevere if we interchange ``domain'' and ``target''. Explicitly, one might ask if we can replace Postnikov space with polyGEM or finite CW complex with $\F$-finite space in the following result.

\begin{thm}\label{thmNoPhant}
A finite CW complex $X$ is not the domain of essential phantom maps. Dually, a Postnikov space of finite type is not the target of essential phantom maps. 
\end{thm}

We answer this question negatively in Section \ref{secOtherWay}.  As an application we give a new calculation of the $\F$-cone length of a particular space. 

Through the work of Pavesic \cite{PavesicProducts} we find applications of Corollary \ref{corPhantomIdent} to the calculation of  automorphism groups in the homotopy category. Recall $\Aut(X)$ is the group of self-homotopy equivalences of a space $X$. 
In Section \ref{secAut} we examine $\Aut(X)$ and some related groups for spaces $X$ admitting a homotopy decomposition as a product of an $\F$-finite space with a polyGEM, in some cases yielding (more or less) explicit calculations like the following.

\begin{ex}
Let $n\geq 3$ be an odd integer, write $K=K(\ZZ,n)$ and suppose $W=\bigvee_\alpha S^{n_\alpha}$ is a finite-type $(n+2)$-connected wedge of spheres. 
Then there is a split short exact sequence
\[
0\to\prod H^{n+n_\alpha}(\s^{n_{\alpha}} K;\pi_{k+1}(W)\otimes\RR)\to \Aut(K\times W) \to \Aut(W)\times \ZZ/2\to 0. 
\]
i.e. $\Aut(K\times W)$ is a semi-direct product of $\prod H^{n+n_\alpha}(\s^{n_{\alpha}} K;\pi_{k+1}(W)\otimes\RR)$ with $\Aut(W)\times \ZZ/2$.
\end{ex}

\paragraph{Acknowledgements}

The author was introduced to phantom maps and resolving classes by his thesis advisor, Jeff Strom, who provided invaluable feedback on an earlier draft of this manuscript and identified a critical flaw. Thanks are due to Jerome Scherer for pointing the author to the existing body of literature on polyGEMs.

%

\section{Background}\label{secBackground}

 \subsection{Finiteness Conditions}

We begin by describing an invariant known as cone length \cite{Cornea199495}, 
or strong category \cite{fox},\cite{ganea1967}. We prefer the first term. 
A \textbf{length $n$ cone decomposition} of a space $Y$ is a sequence of cofiber sequences 
\begin{equation}
A_i\to Y_{(i)}\to Y_{(i+1)} \label{eqConeLength}, \hspace{3pc} i=0,\dots, n
\end{equation}
with $Y_{(0)}\simeq \ast$ and $Y_{(n)}\simeq Y$. 
Familiar examples of length $n$ cone decompositions are CW structures for spaces having the homotopy type of $n$-dimensional CW complexes. 
The \textbf{cone length} $\Cl(Y)$ of a space $Y$ is the least $n$ for which $Y$ admits a length $n$-cone decomposition, allowing for the possibility $\Cl(Y)=\infty$. 


In \cite{ASS} Arkowitz, Stanley, and Strom introduce and study a generalization of cone length. For a collection $\A$ of spaces a length $n$ $\A$-cone decomposition is defined by requiring the spaces in \eqref{eqConeLength} belong to the collection $\A$. 
The $\A$-cone length $\Cl_\A(Y)$ of a space $Y$ is then the least $n$ for which $Y$ admits a length $n$ $\A$-cone 
decomposition. We will say a space $X$ is $\A$-\textbf{finite} if $\Cl_\A(X)<\infty$. Of particular interest will be the $\F$-finite 
spaces, where $\F$ denotes the collection of finite type wedges of spheres. Such spaces can be built in finitely many principal 
cofibrations from finite type wedges of spheres. 

\begin{ex} \label{exFfinite}
\begin{enumerate}
\item Every finite CW complex is an $\F$-finite space, since a CW structure is an $\F$-cone decomposition. 

\item The space $\bigvee_{n=1}^\infty S^n$ is an example of an $\F$-finite space which is not a finite CW complex. 

\item  
 For each odd prime $p$ it is well-known that there is an element $\alpha_p\in\pi_{2p}(S^3)$ of order $p$. Let $\alpha:\bigvee S^{2p} \to S^3$ be the map whose restriction to each summand $S^{2p}$ is $\alpha_p$, and let $X$ be the homotopy cofiber of $\alpha$. Then $\Cl_\F(X)=2$ (see Section \ref{secOtherWay} or \cite{mePhantomsToCoH}, where it is shown that $\Cl(X)=2$, and one has the general inequality $\Cl_\F(X)\geq \Cl(X)$). 
\end{enumerate}
\end{ex}

We now turn to describing a rough dual to the collection of $\F$-finite spaces, the collection of polyGEMs. The study of such spaces goes back to Farjoun \cite{farjoun1995cellular}. The notion of polyGEM we describe here is consistent with that of Dwyer and Farjoun \cite{DwyerFarjoun}, though our presentation will follow Iriye and Kishimoto  \cite{IriyeKishimoto} to highlight similarities with the $\F$-finite spaces, and since we feel this language provides a more convenient means for communicating certain details regarding polyGEMs. 

    A generalized Eilenberg-MacLane space (GEM) is a product of Eilenberg-MacLane spaces $K(A_n,n), n\in\NN, A_n$ a finitely generated Abelian group; these are the spaces of Postnikov-Length one. A \textbf{length} $n$ \textbf{Generalized Postnikov Tower (GPT)} for a space $X$ is a sequence of fiber sequences 
\[
X_{(i+1)}\to X_{(i)} \to P_i, \hspace{3pc} i=0,\dots,n
\]
in which each $P_i$ is a GEM, $X_{(0)}\simeq\ast$ and $X_{(n)}\simeq X$. The \textbf{Postnikov length} $\pl(X)$ of a space $X$ is the least $n$ for which $X$ admits a length $n$ GPT. A space $X$ is called an \textbf{$n$-polyGEM} if $\pl(X)\leq n$, and may be simply referred to as a \textbf{polyGEM} when $n$ is of no particular significance. 

\begin{ex} \label{exPolyGEM}
\begin{enumerate}
\item Every Postnikov space is a polyGEM, since a Postnikov tower is a GPT. 

\item The space $\prod_{n=1}^\infty K(\ZZ,n)$ is an example of a polyGEM which is not a Postnikov space.

\item \cite[Proposition 4.8]{IriyeKishimoto}
For each prime $p$ let $\P_p^1:K(\ZZ,3)\to K(\ZZ/p,2p+1)$ be the composite of the mod $p$ reduction and the Steenrod reduced power operation $\P^1:K(\ZZ/p,3)\to K(\ZZ/p,2p+1)$. Let $X$ be the homotopy fiber of the map 
\[
 K(\ZZ,3)\to \prod K(\ZZ/p,2p+1)
\]
whose projection on the $p$th component of the target is $\P_p^1$. Then $\pl(X)=2$. 
\end{enumerate}
\end{ex}

\subsection{Duality Violating Mechanisms} 


Suppose 
\begin{equation}\label{eqArbFib}
F\to E \to B
\end{equation} is a fiber sequence and $K$ is any space. Applying the functor $\map_\ast(K,-)$ to the fiber sequence \eqref{eqArbFib} yields a fiber sequence. On the other hand, applying the functor $\map_\ast(-,K)$ to the fiber sequence \eqref{eqArbFib} need not produce anything reasonable at all. However, the following theorem of Zabrodsky allows us to draw conclusions about this formally terrible sequence in special cases. 

\begin{zlem}
Suppose $F\to E \to B$ is a (homotopy) fiber sequence. If $Y$ is a space with $\map_\ast(F,Y)\sim \ast$, then the map 
\[
\map_\ast(B,Y)\to \map_\ast(E,Y)
\]
is a weak equivalence. 
\end{zlem}

Dually, if $X\to Y \to C$ is a cofiber sequence one obtains a fiber sequence upon application of the functor $\map_\ast(-,K)$, but one does not expect the sequence resulting from an application of the functor $\map_\ast(K,-)$ to be well-behaved. However, for the proof of Theorem \ref{thmMain} we can get enough information from the theory of resolving classes.

A \textbf{resolving class} is a collection $\R$ of spaces which is closed under weak equivalences and (pointed) homotopy colimits (over compactly indexed diagrams). Explicitly, if $A \in \R$ and there is a weak equivalence $A\to B$ or $B\to A$ then $B\in \R$, and if $F:\I\to \Top$ is a diagram over a category $\I$ with compact classifying space and $F(i)\in \R$ for all $i\in \I$, then $\holim_\I F \in \R$. 
A \textbf{strong resolving class} is a resolving class $\R$ which is also closed under extensions by fibrations, i.e. if $F\to E \to B$ is a fiber sequence with $F,B\in \R$ then $E\in \R$. We record here for future reference that the intersection of a collection of (strong) resolving classes is again a (strong) resolving class. 

\begin{ex}\cite{StromFDinResolving} \label{exResolvingClasses}
\begin{enumerate} 
\item The collection $\{Y\mid Y\sim \ast\}$ is a strong resolving class. 

\item Let $K$ be any space; the collection $\{Y\mid \map_\ast(K,Y)\sim \ast\}$ is a strong resolving class. 

\item More generally, if $F$ is a covariant functor that commutes with homotopy limits, and $K$ is any space then $\{Y\mid \map_\ast(K,F(Y))\sim \ast\}$ is a strong resolving class. 
\end{enumerate} 
\end{ex}

We will need the following variation upon the last example in \ref{exResolvingClasses}. 

\begin{prop}\label{propThingsResolving}
Suppose $K$ is simply connected and write 
\[
\R=\{Y\mid Y \mbox{ is nilpotent and } \map_\ast(K,Y^\wedge)\sim\ast\}.
\]
Then $\R$ is a strong resolving class. 
\end{prop}

In the interest of clarity, we remark that profinite completion does not, in general, commute with the formation of homotopy limits. On the other hand, Dror Farjoun has shown that this failure of commutativity is mild if the homotopy limit is indexed by a small enough category:

\begin{thm}\cite{FarjounCompletingLimits}\label{thmDF}
If $\N$ is a diagram of finite type spaces over a category $\I$ with compact classifying space, then the natural comparison map 
\[
(\holim_\I \N)^\wedge_p \to \holim_\I \N^\wedge_p
\]
has discrete fibers over each component of its target. 
\end{thm}

\begin{proof}[Proof of Proposition \ref{propThingsResolving}]
That $\R$ is closed under weak equivalence follows from the observation, due to Quick \cite{Quick}, that profinite completion preserves weak equivalences. 

Suppose $\N$ is a diagram of nilpotent spaces over a category $\I$ with $\N(i)\in\R$ for all $i\in\I$. Write $Q$ for the homotopy fiber of the natural comparison map $(\holim \N)^\wedge\lra\xi \holim (\N^\wedge)$, and consider the fiber sequence 
\[
\map_\ast(K,Q)\to \map_\ast(K,(\holim\N)^\wedge)\lra{\xi^\wedge_\ast} \map_\ast(K,\holim (\N^\wedge))
\]
Note that 
\[
\map_\ast(K,\holim (\N)^\wedge)\simeq \holim \map_\ast(K,\N^\wedge))\sim\ast,
 \]
and, in particular, $\map_\ast(K,\holim (\N^\wedge))$ is connected. 

According to Sullivan \cite{sullivan1974}, since $X$ is nilpotent of finite type we have a natural equivalence $X^\wedge \simeq \prod_p X^\wedge_p$. So, from Theorem \ref{thmDF} we infer $Q$ is a product of discrete spaces. Since $K$ is connected, $\map_\ast(K,Q)\sim \ast$. Since $\map_\ast(K,\holim (\N^\wedge))$ is connected, we infer $\xi^\wedge_\ast$ is a weak equivalence, and $\map_\ast(K,(\holim\N)^\wedge)\sim \ast$. 
Thus $\R$ is closed under the formation of homotopy limits. 

Finally, suppose $X\to Y \lra p Z$ is a fiber sequence with $X,Z\in\R$. Consider the induced fiber sequence 
\begin{equation}\label{diagIndFib}
P\to Y^\wedge\lra{p^\wedge} Z^\wedge. \notag
 \end{equation}
 Applying the functor $\map_\ast(K,-)$ we obtain another fiber sequence 
 \[
 \map_\ast(K,P)\to \map_\ast(K,Y^\wedge)\lra{(p^\wedge)_\ast}\map_\ast(K,Z^\wedge)\sim \ast. 
 \]
 If we can show $\map_\ast(K,P)\sim \ast$, we can conclude $(p^\wedge)_\ast$ is a weak equivalence in light of the observation that $\map_\ast(K,Z^\wedge)$ is connected. It will then follow that $\map_\ast(K,Y^\wedge)\sim \ast$, which is to say $Y\in\R$. 

  To this end, 
write $X_0$ for the basepoint component of $X$. Then since $K$ is connected $\map_\ast(K,X)\simeq\map_\ast(K,X_0)$. According to May and Ponto \cite[Proposition 11.2.5]{MayPonto} the map $X_0\to P_0$ is a profinite completion of $X_0$, i.e. $P_0\simeq (X_0)^\wedge$. So, we have 
\[
\map_\ast(K,P)\simeq \map_\ast(K,P_0)\simeq \map_\ast(K,(X_0)^\wedge)\simeq \map_\ast(K,(X^\wedge)_0)\simeq \map_\ast(K,X^\wedge)\sim \ast
\]
since $X\in\R$. 
\end{proof}

Strom \cite{StromFDinResolving} shows that resolving classes possess a number of formally implausible closure properties. For example, the following theorem shows that a strong resolving class is closed under the formation of particular homotopy \emph{co}limits! Before stating the result, we establish some notation. 
Given collections $\A$ and $\B$ of simply-connected spaces we write
\begin{align*}
\A\wedge \B &= \{ A \wedge B \mid A \in \A, B\in \B\} \\
\s \A &= \{\s A \mid A \in \A \} \\
\A^\vee &= \{ \mbox{all finite type wedges of spaces in } \A\}. 
\end{align*}

\begin{thm}\cite[Corollary 11]{StromFDinResolving}\label{thmResolvingHammer} 
If $\R$ is a strong resolving class with $\s \A^\vee \subseteq \R, \A \wedge \A \subseteq \A$ and $\s \A \subseteq \A$, then $\R$ contains every space $K$ with $\Cl_{\A^\vee}(K)<\infty$. 
\end{thm}


\section{Proof of Theorem \ref{thmMain}} \label{secMainResult} 

We begin by establishing a preliminary result. 

\begin{prop}\label{propStep1}
If $K$ is a simply-connected polyGEM, and $Y$ is a finite complex, then $\map_\ast(K,Y^\wedge)\sim\ast$. 
\end{prop}

\begin{proof}
The argument is by induction on the Postnikov length of $K$. For the base case we consider a GEM $K(A)\simeq \prod_n K(A_n,n)$. Recall that for a collection $\{X_\alpha\}$ of spaces, the weak product $\tilde{\prod} X_\alpha$ is the colimit of all finitely indexed subproducts of the categorical product $\prod X_\alpha$. If for each $i$ there are only finitely many $\alpha$ for which $\pi_i(X_\alpha) \neq 0$  the canonical comparison map $\tilde{\prod} X_\alpha \to \prod X_\alpha$ is a weak equivalence. Of course $K(A)$ meets the stated requirement on homotopy groups, and so to show $\map_\ast(K(A),Y^\wedge)\sim \ast$ for $Y$ a finite CW complex, it suffices to show $\map_\ast(\tilde{\prod} K(A_n,n),Y^\wedge)\sim \ast$. 
Of course, if $P$ is any finitely indexed subproduct of $K(A)$, then $P$ is a Postnikov space, and so $\map_\ast(P,Y^\wedge)\sim\ast$ by Theorem \ref{thmZabrodsky}. Finally, since $\map_\ast(\colim X_i,Y)\simeq \lim \map_\ast(X_i,Y)$ we conclude $\map_\ast(\tilde{\prod} K(A_n,n),Y^\wedge)\sim \ast$, which establishes the base case.

For the inductive step, suppose that when $\pl(Z)\leq n$ we have $\map_\ast(Z, Y^\wedge)\sim \ast$ for all finite complexes $Y$, and let $X$ be a space with $\pl(X)=n+1$. Then we have a fiber sequence 
\[
X \to Z \to K 
\]
with $\pl(Z)=n$ and $\pl(K)=1$, which gives rise to a fiber sequence 
\begin{equation}\label{PfFiberSeq}
\om K \to X \to  Z. 
 \end{equation}
 Applying the Zabrodsky lemma to the fibration \eqref{PfFiberSeq} we see that since $\map_\ast( \om K, Y^\wedge)\sim \ast$ the induced map 
\[
\ast\sim \map_\ast(Z,Y^\wedge)\to \map_\ast(X, Y^\wedge)
\]
is a weak equivalence. 
\end{proof} 

To complete the proof of Theorem \ref{thmMain} we turn to the theory of resolving classes. 
Consider the collection 
\[
\R=\{Y\mid \map_\ast(K,Y^\wedge)\sim\ast \mbox{ for all simply-connected polyGEMs } K\}.
\]
Since the intersection of strong resolving classes is a strong resolving class, by Proposition \ref{propThingsResolving} we infer $\R$ is a strong resolving class. 
We now appeal to Theorem \ref{thmResolvingHammer}. Let $\A$ be the collection of spheres, so $\A^\vee =\F$. Then clearly $\s\A \subseteq \A \wedge \A \subseteq \A$. It remains to show that $\s\A^\vee\subseteq \R$. 

This argument is essentially due to Jeff Strom \cite[Theorem 8]{StromFDinResolving}. 
First we define a partial order on $\s\A^\vee$. For $X,Y\in \s\A^\vee$ we say $X<Y$ if $\conn(Y)<\conn(X)$ or if $\conn(X)=\conn(Y)$ and the rank of $\pi_n(Y)$ is less than the rank of $\pi_n(X)$. 

Let $W\in \s \A^\vee$, with $\conn(W)=n-1$. Write  $W=\s A \vee S^n$ and note $\s A <W$. 
 Gray  \cite{GrayOnHMThm} has shown the homotopy fiber $W_1$ of the collapse map $\s A \vee S^n  \to S^n$ is homotopy equivalent to $\s A \wedge \left(\bigvee_k S^{nk}\right)$. It follows that $W_1\in \s \A^\vee$, and, moreover, $W_1<W$. Inductively, this defines a tower of principal fibrations
\[
\dots \to W_n \lra{f_n} \dots \lra{f_3} W_2 \lra{f_2}  W_1\lra{f_1}  W
\]
induced by maps $W_n\to S^{g(n)}$, where $g(n)$ is the dimension of the smallest sphere summand of $W_n$,
in which $W_n < W_{n-1} \in \s\A^\vee$ for all $n$. By the definition of $<$ we have $\lim_{n\to\infty} \conn(W_n)=\infty$ and so $\holim W_n\sim \ast$. 

Now, since $W$ is simply-connected and $W_i<W$ for all $i$ each $W_i$ is simply-connected. It follows that $S^{g(i)}$ is simply-connected for all $i$, being a summand of $W_i\in\s\A^\vee$, so the fiber sequences
\[
W_{i+1}\to W_{i}\to S^{g(i)}
\]
are preserved under completion \cite[Theorem 11.2.6]{MayPonto}.  Let $K$ be a simply-connected polyGEM,
and consider the fiber sequences 
\[
\map_\ast(K,W_{i+1}^\wedge)\to \map_\ast(K,W_i^\wedge)\to \map_\ast(K,(S^{g(i)})^\wedge).
\]
The right-most space is weakly contractible by Proposition \ref{propStep1}, so $\map_\ast(K,f_i^\wedge)$ is a weak equivalence for all $i$. It follows that the composite maps $g_i=f_i\circ \dots\circ f_2\circ f_1:W_i\to W$ induce weak equivalences
\[
\map_\ast(K,W_i^\wedge)\sim \map_\ast(K,W^\wedge)
\]
for all $i$. 
Then we have 
\begin{align}
\map_\ast(K,W^\wedge) & \sim \holim \map_\ast(K,W^\wedge) \\\notag
&\sim \holim \map_\ast(K,W_i^\wedge) \\\notag
& \sim \map_\ast(K, \holim (W_i^\wedge)).
\end{align}
Now, for any space $X$ one has $\conn(X^\wedge)\geq \conn(X)$, so $\lim_{n\to\infty} \conn W_i^\wedge = \infty$ and so $\holim W_n^\wedge \sim \ast$. So, we conclude 
\[
\map_\ast(K,W^\wedge) \sim \map_\ast(K, \holim (W_i^\wedge)) \sim \map_\ast(K, \ast)=\ast
\]
and $W\in\R$. 

Having shown that $\s \A^\vee\subseteq \R$, Theorem \ref{thmResolvingHammer} implies $\R$ contains all $\F$-finite spaces, which completes the proof of Theorem \ref{thmMain}.

\paragraph{More General Notions of PolyGEM} More general notions of polyGEMs appear in the literature. One particularly general formulation appears in the work of Chach\'{o}lski, Farjoun, Flores, and Scherer \cite{MR3416113}: 
A 1-polyGEM is a still a GEM, but an $n$-polyGEM is a space weakly equivalent to a retract of the homotopy fiber of a map $X\to K$ where $X$ is an $(n-1)$-polyGEM and $K$ is a GEM. This class has the benefit of being totally characterized by a relation to a modified Bousfield-Kan completion tower. We decided to present the notion above because it is roughly dual to the notion of $\F$-finite spaces, but our arguments apply in more general case as well,  since if $Z$ is a retract of a polyGEM $K$ as defined in Section \ref{secBackground}, then $\map_\ast(Z,Y)$ is a retract of $\map_\ast(K,Y)$ for all spaces $Y$.

\section{Role Reversal}\label{secOtherWay}

Theorem \ref{thmMain} indicates that from the perspective of phantom map theory polyGEMs behave similarly to Postnikov spaces as domains, and as targets $\F$-finite spaces behave similarly to finite complexes. One might wonder how polyGEMs behave in the target and $\F$-finite spaces in the domain. As a means for motivation we restate Theorem \ref{thmNoPhant} from the introduction. 

\begin{thm}\label{thmRep}
A finite CW complex $X$ is not the domain of essential phantom maps. Dually, a Postnikov space $Y$ of finite type is not the target of essential phantom maps. 
\end{thm}


According to Gray and McGibbon \cite[Corollary 2.1]{univPhantMap} if $\s X$ is equivalent to a bouquet of finite complexes, then $X$ is not the domain of essential phantom maps, and according to McGibbon \cite[Theorem 3.20(iv)]{handbook} if $\om Y$ is a product of Postnikov spaces then $Y$ is not the target of essential phantom maps. So, Theorem \ref{thmRep} remains true if we replace $X$ with a space of $\F$-cone length one or $Y$ with a space of Postnikov length one. However, the following examples show that Theorem \ref{thmNoPhant} cannot be extended to general $\F$-finite spaces, nor to polyGEMs. 

\begin{ex}\cite[Example 3.13]{handbook}
For each prime $p$ let $\alpha_p:S^{2p}\to S^3$ represent an element of $\pi_{2p} S^3$ of order $p$. Let $\alpha:\bigvee_{p\geq 3} S^{2p}\to S^3$, where the wedge is taken over all odd primes, be the map whose restriction to each summand $S^{2p}$ is $\alpha_p$. Let $X$ be the homotopy cofiber of $\alpha$. Evidently $\Cl_\F(X)\leq 2$, but there are essential phantom maps $X\to S^4$. 
\end{ex}

\begin{rem}
In particular this shows $\Cl_\F(X)=2$ since if $\Cl_\F(X)=1$ then $X\in\F$ and $X$ is a wedge of finite complexes.
\end{rem}

\begin{ex}\cite[Proposition 4.8]{IriyeKishimoto}
For each prime $p$ let $\P_p^1:K(\ZZ,3)\to K(\ZZ/p,2p+1)$ be the composite of the mod $p$ reduction and the Steenrod reduced power operation $\P^1:K(\ZZ/p,3)\to K(\ZZ/p,2p+1)$. Let $Y$ be the homotopy fiber of the map 
\[
 K(\ZZ,3)\to \prod K(\ZZ/p,2p+1)
\]
whose projection on the $p$th component of the target is $\P_p^1$. Then $\pl(Y)=2$ and 
\[
\Ph(K(\ZZ/2),Y)\cong \ZZ^\wedge/\ZZ\oplus\bigoplus \ZZ/p^\infty. 
\]
\end{ex}

\section{Automorphisms} \label{secAut}

Here we demonstrate some applications of Theorem \ref{thmMain} and Corollary \ref{corPhantomIdent} to the study of the group $\Aut(X)$ of homotopy classes of (based) self-homotopy equivalences of a space $X$. We will be interested in the case where $X$ splits, up to homotopy, as a product of a polyGEM with an $\F$-finite space.

Write $i_X,i_Y$ for the inclusions of $X$ and $Y$ into $X\times Y$ and write $p_X,p_Y$ for the projections of $X\times Y$ onto $X$ and $Y$, respectively. For a map $f:X\times Y \to X\times Y$ and $I,J\in\{X,Y\}$ let $f_I:X\times Y\to I$ be the composite $f_I=p_I\circ f$ and let $f_{IJ}$ be the composite $I\to X\times Y \lra f X\times Y \to J$. Define a subset $\Aut_X(X\times Y)\subseteq \Aut(X\times Y)$ to be the set containing those maps $f:X\times Y \to X\times Y$ that fix $X$, which is to say maps having the form $(p_X,f_Y)$. The subset $\Aut_Y(X\times Y)$ is defined analogously. It happens \cite[Proposition 2.3]{PavesicProducts} that $\Aut_X(X\times Y)$ and $\Aut_Y(X\times Y)$ are subgroups of $\Aut(X\times Y)$. 

We say the self-equivalences of $X\times Y$ can be \textbf{diagonalized} if for every $f\in\Aut(X\times Y)$ one has $f_{XX}\in \Aut(X)$ and $f_{YY}\in\Aut(Y)$. For subgroups $A,B$ of a group $G$ let $A\cdot B=\{a\cdot b\mid a\in A, b\in B\}$. The key reason one might be interested in the diagonalizability of self-homotopy equivalences is the following. 

\begin{thm}\cite[Theorem 2.5]{PavesicProducts}\label{pav1}
If $X$ and $Y$ are connected CW-complexes and the self-equivalences of $X\times Y$ are diagonalizable, then 
\[
\Aut(X\times Y)= \Aut_X(X\times Y)\cdot \Aut_Y(X\times Y)
\]
\end{thm}

\begin{prop}\cite[Proposition 2.1]{PavesicProducts}\label{pav2}
Assume that for every $n>0$ and every pair of maps $X\lra f Y, Y\lra g X$ at least one of the induced maps $\pi_n(f),\pi_n(g)$ is trivial. Then self-homotopy equivalences of $X\times Y$ can be diagonalized. 
\end{prop}

By Theorem \ref{thmMain} self-homotopy equivalences of $X\times Y$ can be diagonalized whenever $X$ is a polyGEM and $Y$ is $\F$-finite, since a phantom map induces the zero map on homotopy groups.


Using Theorem \ref{thmMain} and Corollary \ref{corPhantomIdent} we are able to make some (more or less) explicit calculations. 
Before turning to these calculations, we collect a few results; the first is a classical result due to Booth and Heath, and the second is a more modern observation due to Pavesic. 

\begin{thm}\cite{BoothHeath}\label{BH}
If $X$ and $Y$ are connected CW complexes and if $[Y,\aut_1(X)]=\ast$ then there is a split short exact sequence 
\[
0\to [X,\aut_1(Y)]\to \Aut(X\times Y)\to \Aut(X)\times \Aut(Y)\to 0. 
\]
\end{thm}

\begin{prop}\cite[Proposition 4.2]{PavesicProducts} \label{propCoHSeq}
If $Y$ is a co-H-space then there is an exact sequence of groups and pointed sets 
\begin{align*}
\dots \to [\s^2 X, Y]\to[\s X \wedge Y, Y] &\to [\s X,\aut_1(Y)]\to [\s X, Y]  \to [X\wedge Y,Y]\\ & \to [X,\aut_1(Y)]
 \to [X,Y] \to [X,B\aut_1^\ast(Y)]. 
\end{align*}
\end{prop}

%
%
%
%

\begin{ex}\label{exA}
Let $n\geq 3$ be an odd integer, write $K=K(\ZZ,n)$ and suppose $W=\bigvee_\alpha S^{n_\alpha}$ is a finite-type wedge of spheres with $\conn(W)\geq n+2$.
Then there is a split short exact sequence
\[
0\to\prod H^{n+n_\alpha}(\s^{n_{\alpha}} K;\pi_{k+1}(W)\otimes\RR)\to \Aut(K\times W) \to \Aut(W)\times \ZZ/2\to 0. 
\]
i.e. $\Aut(K\times W)$ is a semi-direct product of $\prod H^{n+n_\alpha}(\s^{n_{\alpha}} K;\pi_{k+1}(W)\otimes\RR)$ with $\Aut(W)\times \ZZ/2$.
\end{ex}

\begin{proof}
According to May \cite{MaySimplicialObjects} we have $\aut_1(K)\simeq K$ so $[W,\aut_1(K)]\cong\bigoplus_\alpha \pi_{n_\alpha} K =0$, so Theorem \ref{BH} can be applied in this context. We also have $\Aut(K)\cong \Aut(\ZZ)\cong \ZZ/2$, and so we have a split short exact sequence 
\[
0\to[K,\aut_1(W)]\to \Aut(K\times W)\to \Aut(W)\times \ZZ/2\to 0. 
\]
Now, by Proposition \ref{propCoHSeq} there is an exact sequence
\[
[\s K,W]\to [W\wedge K,W]\to [K,\aut_1(W)]\to [K,W]
\]
which, in light of Theorem \ref{thmMain} is the same as the sequence
\[
 \Ph(\s K,W)\to \Ph(W\wedge K,W) \to [K,\aut_1(W)]\to \Ph(K,W) .
\]
But by Corollary \ref{corPhantomIdent} we have $\Ph(K,W)\cong H^n(K;\pi_{n+1}(W)\otimes\RR)=0$ and $\Ph(\s K,W)\cong H^{n+1}(\s K;\pi_{n+2}(W)\otimes\RR)=0$, since the rational cohomology of $K$ is concentrated in degree $n$ and $W$ is $(n+2)$-connected. It follows that $[K,\aut_1(W)]\cong\Ph(W\wedge K,W)$. Now $W\wedge K\simeq \bigvee \s^{n_\alpha}K$, so again by Corollary \ref{corPhantomIdent} we infer 
\[
\Ph(W\wedge K,W)\cong \prod H^{n+n_\alpha}(\s^{n_\alpha} K; \pi_{n+n_\alpha+1}(W)\otimes\RR)
\]
which establishes Example \ref{exA}
\end{proof}

%

There is also particular interest in the subgroups $\Aut_{\sharp N}(X)$ of the group $\Aut(X)$ consisting of those self-equivalences inducing the identity map on homotopy groups $\pi_n(X)$ for $n\leq N$. There are subgroups $\Aut_{X,\sharp N}(X\times Y), \Aut_{Y,\sharp N}(X\times Y)$ of $\Aut_{\sharp N}(X)$ and analogs of Theorem \ref{pav1} and Proposition \ref{pav2} that can be used to study the group  $\Aut_{\sharp N}(X)$. There is also a generalization of an analog of Theorem \ref{BH}: 

\begin{prop}\cite[Proposition 2.9.d)]{PavesicProducts}\label{propIdentSubgrps}
If $X$ is connected then there is a split short exact sequence 
\[
0\to K_N(X,Y)\to \Aut_{X,\sharp N}(X\times Y)\to \Aut_{\sharp N}(Y)\to 0
\]
where $K_N(X,Y)$ is the subgroup of $[X,\aut_1(Y)]$ consisting of classes $[f]$ such that the composition 
\[
X\lra f \aut_1(Y)\lra{\mathrm{ev}} Y 
\]
induces the zero map on $\pi_n$ for all $n\leq N$. 
\end{prop}

In particular, if $X$ is a polyGEM and $Y$ is $\F$-finite, then $K_N(X,Y)=[X,\aut_1(Y)]$ for all $N$. 

\begin{ex}
With $K$ and $W$ as in Example \ref{exA} and $N\geq n$ there is a split short exact sequence
\[
0\to \prod H^{n+n_\alpha}(\s^{n_\alpha} K; \pi_{n+n_\alpha+1}(W)\otimes\RR)\to \Aut_{\sharp N}(K\times W) \to \Aut_{\sharp N}(W)\to 0. 
\]
\end{ex}

This is established in much the same way as Example \ref{exA}, noting that $\Aut_{\sharp N}(K)=0$ for $N\geq n$.

\bibliographystyle{plain}

\bibliography{DissertationBiblio}{}

\begin{thebibliography}{10}

\bibitem{ASS}
Martin Arkowitz, Donald Stanley, and Jeffrey Strom.
\newblock The cone length and category of maps: Pushouts, products and
  fibrations.
\newblock {\em Bulletin of the Belgian Mathematical Society - Simon Stevin},
  11(4):517--545, 12 2004.

\bibitem{BoothHeath}
P.I. Booth and P.R. Heath.
\newblock On the group ɛ(x×y) and ɛ b b (x×by).
\newblock In RenzoA. Piccinini, editor, {\em Groups of Self-Equivalences and
  Related Topics}, volume 1425 of {\em Lecture Notes in Mathematics}, pages
  17--31. Springer Berlin Heidelberg, 1990.

\bibitem{MR3416113}
Wojciech Chach{\'o}lski, Emmanuel~Dror Farjoun, Ram{\'o}n Flores, and
  J{\'e}r{\^o}me Scherer.
\newblock Cellular properties of nilpotent spaces.
\newblock {\em Geom. Topol.}, 19(5):2741--2766, 2015.

\bibitem{Cornea199495}
Octavian Cornea.
\newblock Cone-length and {L}usternik-{S}chnirelmann category.
\newblock {\em Topology}, 33(1):95 -- 111, 1994.

\bibitem{DwyerFarjoun}
W.~G. Dwyer and E.~D. Farjoun.
\newblock Localization and cellularization of principal fibrations.
\newblock In {\em Alpine perspectives on algebraic topology}, volume 504 of
  {\em Contemp. Math.}, pages 117--124. Amer. Math. Soc., Providence, RI, 2009.

\bibitem{farjoun1995cellular}
E.~Farjoun.
\newblock {\em Cellular Spaces, Null Spaces and Homotopy Localization}.
\newblock Cellular Spaces, Null Spaces, and Homotopy Localization. Springer,
  1995.

\bibitem{FarjounCompletingLimits}
Emmanuel~Dror Farjoun.
\newblock Bousfield-{K}an completion of homotopy limits.
\newblock {\em Topology}, 42(5):1083 -- 1099, 2003.

\bibitem{fox}
Ralph~H. Fox.
\newblock On the {L}usternik-{S}chnirelmann category.
\newblock {\em Annals of Mathematics}, 42(2):pp. 333--370, 1941.

\bibitem{ganea1967}
T.~Ganea.
\newblock Lusternik-{S}chnirelmann category and strong category.
\newblock {\em Illinois Journal of Mathematics}, 11(3):417--427, 09 1967.

\bibitem{GrayOnHMThm}
Brayton Gray.
\newblock A note on the {H}ilton-{M}ilnor theorem.
\newblock {\em Topology}, 10(3):199 -- 201, 1971.

\bibitem{univPhantMap}
Brayton Gray and C.A. McGibbon.
\newblock Universal phantom maps.
\newblock {\em Topology}, 32:371--394, 1993.

\bibitem{GraySNT}
Brayton~I. Gray.
\newblock Spaces of the same n-type, for all n.
\newblock {\em Topology}, 5(3):241 -- 243, 1966.

\bibitem{IriyeKishimoto}
Kouyemon Iriye and Daisuke Kishimoto.
\newblock Postnikov towers with fibers generalized {E}ilenberg-{M}ac{L}ane
  spaces.
\newblock {\em Homology, Homotopy and Applications}, 14(1):1--14, 2012.

\bibitem{MaySimplicialObjects}
J.~Peter May.
\newblock {\em Simplicial objects in algebraic topology}.
\newblock Van Nostrand Mathematical Studies, No. 11. D. Van Nostrand Co., Inc.,
  Princeton, N.J.-Toronto, Ont.-London, 1967.

\bibitem{MayPonto}
J.P. May and K.~Ponto.
\newblock {\em More Concise Algebraic Topology: Localization, Completion, and
  Model Categories}.
\newblock Chicago Lectures in Mathematics. University of Chicago Press, 2012.

\bibitem{handbook}
C.A. McGibbon.
\newblock Phantom maps.
\newblock In {\em The Handbook of Algebraic Topology}, pages 1209--1257.
  North-Holland, Amsterdam, 1995.

\bibitem{PavesicProducts}
Petar Pave{\v{s}}i{\'c}.
\newblock Self-homotopy equivalences of product spaces.
\newblock {\em Proc. Roy. Soc. Edinburgh Sect. A}, 129(1):181--197, 1999.

\bibitem{Quick}
G.~Quick.
\newblock Profinite homotopy theory.
\newblock {\em Documenta Math.}, 13, 2008.

\bibitem{RoitbergAut}
J.~{Roitberg}.
\newblock Note on phantom phenomena and groups of self-homotopy equivalences.
\newblock {\em Comment. Math. Helvetici}.

\bibitem{mePhantomsToCoH}
J.~Schwass.
\newblock On phantom maps into co-h-spaces.
\newblock {\em Submitted}, 2016.

\bibitem{StromFDinResolving}
Jeffrey Strom.
\newblock Finite-dimensional spaces in resolving classes.
\newblock {\em Fund. Math.}, 217(2):171--187, 2012.

\bibitem{sullivan1974}
D.~Sullivan.
\newblock Genetics of homotopy theory and the adams conjecture.
\newblock {\em Ann. of Math.}

\bibitem{zabrodsky}
Alex Zabrodsky.
\newblock On phantom maps and a theorem of {H}. {M}iller.
\newblock {\em Israel J. Math}, 58:129--143, 1987.

\end{thebibliography}
%

%

\end{document}